\documentclass[12pt]{amsart}
\usepackage{amssymb}
\usepackage{verbatim}
\usepackage{hyperref}
\usepackage{color}

\begin{document}

\newtheorem{theorem}{Theorem}    
\newtheorem{proposition}[theorem]{Proposition}
\newtheorem{conjecture}[theorem]{Conjecture}
\def\theconjecture{\unskip}
\newtheorem{corollary}[theorem]{Corollary}
\newtheorem{lemma}[theorem]{Lemma}
\newtheorem{sublemma}[theorem]{Sublemma}
\newtheorem{fact}[theorem]{Fact}
\newtheorem{observation}[theorem]{Observation}
\theoremstyle{definition}
\newtheorem{definition}{Definition}
\newtheorem{notation}[definition]{Notation}
\newtheorem{remark}[definition]{Remark}
\newtheorem{question}[definition]{Question}
\newtheorem{questions}[definition]{Questions}
\newtheorem{example}[definition]{Example}
\newtheorem{problem}[definition]{Problem}
\newtheorem{exercise}[definition]{Exercise}

\numberwithin{theorem}{section}
\numberwithin{definition}{section}
\numberwithin{equation}{section}

\def\reals{{\mathbb R}}
\def\torus{{\mathbb T}}
\def\heis{{\mathbb H}}
\def\integers{{\mathbb Z}}
\def\rationals{{\mathbb Q}}
\def\naturals{{\mathbb N}}
\def\complex{{\mathbb C}\/}
\def\distance{\operatorname{distance}\,}
\def\support{\operatorname{support}\,}
\def\dist{\operatorname{dist}\,}
\def\Span{\operatorname{span}\,}
\def\degree{\operatorname{degree}\,}
\def\kernel{\operatorname{kernel}\,}
\def\dim{\operatorname{dim}\,}
\def\codim{\operatorname{codim}}
\def\trace{\operatorname{trace\,}}
\def\Span{\operatorname{span}\,}
\def\dimension{\operatorname{dimension}\,}
\def\codimension{\operatorname{codimension}\,}
\def\nullspace{\scriptk}
\def\kernel{\operatorname{Ker}}
\def\ZZ{ {\mathbb Z} }
\def\p{\partial}
\def\rp{{ ^{-1} }}
\def\Re{\operatorname{Re\,} }
\def\Im{\operatorname{Im\,} }
\def\ov{\overline}
\def\eps{\varepsilon}
\def\lt{L^2}
\def\diver{\operatorname{div}}
\def\curl{\operatorname{curl}}
\def\etta{\eta}
\newcommand{\norm}[1]{ \|  #1 \|}
\def\expect{\mathbb E}
\def\bull{$\bullet$\ }
\def\det{\operatorname{det}}
\def\Det{\operatorname{Det}}
\def\multiR{\mathbf R}
\def\bestA{\mathbf A}
\def\Apq{\mathbf A_{p,q}}
\def\Apqr{\mathbf A_{p,q,r}}
\def\rank{\mathbf r}
\def\diameter{\operatorname{diameter}}
\def\bp{\mathbf p}
\def\bff{\mathbf f}
\def\bg{\mathbf g}
\def\essd{\operatorname{essential\ diameter}}

\def\mab{\max(|A|,|B|)}
\def\t2{\tfrac12}
\def\tatb{tA+(1-t)B}

\newcommand{\abr}[1]{ \langle  #1 \rangle}

\def\doublesymm{\natural}

\newcommand{\Norm}[1]{ \Big\|  #1 \Big\| }
\newcommand{\set}[1]{ \left\{ #1 \right\} }
\def\one{{\mathbf 1}}
\newcommand{\modulo}[2]{[#1]_{#2}}

\def\scriptf{{\mathcal F}}
\def\scripts{{\mathcal S}}
\def\scriptq{{\mathcal Q}}
\def\scriptg{{\mathcal G}}
\def\scriptm{{\mathcal M}}
\def\scriptb{{\mathcal B}}
\def\scriptc{{\mathcal C}}
\def\scriptt{{\mathcal T}}
\def\scripti{{\mathcal I}}
\def\scripte{{\mathcal E}}
\def\scriptv{{\mathcal V}}
\def\scriptw{{\mathcal W}}
\def\scriptu{{\mathcal U}}
\def\scripta{{\mathcal A}}
\def\scriptr{{\mathcal R}}
\def\scripto{{\mathcal O}}
\def\scripth{{\mathcal H}}
\def\scriptd{{\mathcal D}}
\def\scriptl{{\mathcal L}}
\def\scriptn{{\mathcal N}}
\def\scriptp{{\mathcal P}}
\def\scriptk{{\mathcal K}}
\def\scriptP{{\mathcal P}}
\def\scriptj{{\mathcal J}}
\def\scriptz{{\mathcal Z}}
\def\frakv{{\mathfrak V}}
\def\frakG{{\mathfrak G}}
\def\frakA{{\mathfrak A}}
\def\frakB{{\mathfrak B}}
\def\frakC{{\mathfrak C}}

\author{Michael Christ}
\address{
        Michael Christ\\
        Department of Mathematics\\
        University of California \\
        Berkeley, CA 94720-3840, USA}
\email{mchrist@math.berkeley.edu}
\thanks{Research supported in part by NSF grant DMS-0901569.}


\date{July 19, 2012.}

\title {Near Equality In The Brunn-Minkowski Inequality} 
\begin{abstract}
A pair $(A,B)$ of subsets of $\reals^d$ which nearly realizes equality in the Brunn-Minkowski inequality, 
must nearly coincide with a homothetic pair of convex sets.
This was previously known for $d=1,2$, and is shown here to hold in all dimensions.
\end{abstract}
\maketitle

\section{Introduction}

This is another in a series of papers 
concerned with the interplay between linear structure, arithmetic progressions, sumsets,
convexity, sharp analytic inequalities, symmetry, and affine-invariant geometry of Euclidean spaces.
\cite{christradon} combines symmetry considerations with a characterization of equality
in the Riesz-Sobolev rearrangement inequality to find extremizers for an inequality
for the Radon transform.
\cite{youngdiscrete} characterizes functions which nearly realize equality in Young's convolution
inequality for torsion-free discrete groups. 
\cite{christrieszsobolev} partially characterizes cases of near equality in the Riesz-Sobolev
inequality; it turns on a connection between this problem and sumsets, and relies
on a characterization of near equality in the one-dimensional Brunn-Minkowski inequality
concerned with sizes of sumsets.
\cite{christyoungest} characterizes cases of near equality in Young's convolution equality
for Euclidean groups, using the result of \cite{christrieszsobolev} as the pivotal input
for the one-dimensional case, and a description of near-solutions of the functional equation
$f(x)+g(y)=h(x+y)$ to extend the result to higher dimensions.
\cite{christbmtwo} characterizes cases of near equality in the Brunn-Minkowski inequality
for $\reals^2$, relying on the corresponding result for $\reals^1$ and the approximate
functional equation result from \cite{christyoungest}.
The present paper is concerned with the higher-dimensional case of this same problem.

Let $A,B\subset\reals^d$ be sets.  Their sumset $A+B$ is defined to be
\[A+B=\set{a+b: a\in A \text{ and } b\in B}.\] 
The Brunn-Minkowski inequality states that for any nonempty Borel sets $A,B\subset\reals^d$, 
\begin{equation} \label{eq:BMadd} |A+B|^{1/d}\ge |A|^{1/d}+|B|^{1/d},  \end{equation}
where $|S|$ denotes the Lebesgue measure of a set $S$.
If $A,B$ have positive, finite Lebesgue measures then
equality holds if and only if there exists a homothetic pair $K,L$
of convex sets such that $A\subset K$, $B\subset L$,
$|K\setminus A|=0$, and $|L\setminus B|=0$ \cite{henstock},\cite{HO}.
There are various formulations for Lebesgue measurable sets; we restrict attention
to Borel sets merely for the sake of simplicity of statements. 

Our main result characterizes pairs of sets which realize near equality in this inequality.

\begin{theorem} \label{thm:BMadd}
Let $d\ge 1$.
For any compact subset $\Lambda\subset(0,\infty)$ 
there exists a function $(0,\infty)\owns \delta\mapsto \eps(\delta)\in(0,\infty)$
satisfying $\lim_{\delta\to 0}\eps(\delta)=0$ with the following property.
Let $A,B\subset\reals^d$ be Borel sets with  positive Lebesgue measures
satisfying $|A|/|B|\in \Lambda$.  If
\begin{equation} |A+B|^{1/d}\le |A|^{1/d}+ |B|^{1/d}+\delta\mab^{1/d} \end{equation}
then there exist a compact convex set $\scriptk\subset\reals^d$,
scalars $\alpha,\beta\in\reals^+$, and vectors $u,v\in\reals^d$ such that
\begin{equation} A \subset \alpha \scriptk+u \text{ and } 
|(\alpha\scriptk+u)\setminus A|<\eps(\delta)\mab \end{equation}
while
\begin{equation} B \subset \beta\scriptk+v \text{ and } 
|(\beta\scriptk+v)\setminus B|<\eps(\delta)\mab. \end{equation}
\end{theorem}

The Brunn-Minkowski inequality is nontrivial even for convex sets, and was first formulated 
in that restricted context, but the main issue here is to establish that arbitrary sets satisfying
the hypothesis are nearly convex. Our analysis sheds little or no light on the problem of 
finding a quantitative upper bound on the distance from a pair of convex sets satisfying the hypothesis, 
to a pair of homothetic convex sets, in terms of $\delta$.

The case $d=2$ was treated in an earlier paper \cite{christbmtwo}, with which the present work
has much in common. Each involves a
proof by induction on the dimension, regarding $\reals^d$ as $\reals^{d-1}\times\reals^1$ and
using the theorem for both factors $\reals^{d-1}$ and $\reals^1$ to treat $\reals^d$.
When $d=2$, both factors become $\reals^1$, and the Brunn-Minkowski equality for each factor takes
the simpler form $|t\scripta+(1-t)\scriptb|\ge t|\scripta|+(1-t)|\scriptb|$ for $t\in(0,1)$. 
The proof given in \cite{christbmtwo} for $d=2$ relies on this simplification, and does not seem
to extend directly to higher dimensions.

The main new element in the present work is the use of symmetrization.
To a pair $(A,B)$ we associate equimeasurable partial symmetrizations $(A^\doublesymm,B^\doublesymm)$;
these are partial symmetrizations in the sense that they are in general more symmetric
than $A,B$ but are not necessarily ellipsoids.
These have several key properties:
(i) $|A^\star+B^\star| \le |A+B|$; 
(ii) $A^\doublesymm,B^\doublesymm$ have sufficient structure that near equality in the 
Brunn-Minkowski inequality for them is readily shown to imply the conclusion of our theorem;
(iii) because the symmetrization is merely partial, $A^\doublesymm,B^\doublesymm$
retain sufficient information that their near coincidence with convex sets implies useful information
about $A,B$ themselves. Consequently, the induction hypothesis in $\reals^{d-1}$
is applied only to pairs of sets $(\scripta,\scriptb)$ known to have nearly equal measures, 
so that the distinction between $t|\scripta|+(1-t)|\scriptb|$ 
and $(t|\scripta|^{1/(d-1)}+(1-t)|\scriptb|^{1/(d-1)})^{d-1}$ becomes inconsequential.
The symmetrization $A^\doublesymm$ was used by Burchard \cite{burchard} to characterize
equality in the Riesz-Sobolev rearrangement inequality, and has also been employed in other aspects
of symmetrization theory \cite{liebloss}.

\begin{notation}
Throughout the analysis, we identify $\reals^d$ with $\reals^{d-1}\times\reals^{1}$. 
$|S|$ denotes the Lebesgue measure of a subset $S$ of $\reals^{d-1}$ of $\reals^d$, or of $\reals^1$, 
or of a product of two Euclidean spaces, according to context. 
If $A\subset\reals^d$, then for $x\in\reals^1$, $A_x$ denotes the set
\[A_x=\set{y\in\reals^{d-1}: (x,y)\in A}\]
while
\[ \scripta(s)=\set{x\in\reals^{d-1}: |A_x|>s};  \]
more precisely, the set of all $x$ for which $A_x$ is Lebesgue measurable
and $|A_x|>s$.
The projection $\pi:\reals^d\to\reals^{d-1}$ is the mapping $(x,y)\mapsto x$.

$o_\delta(1)$ denotes any quantity which tends to zero as $\delta\to 0$. 
Such a quantity is permitted to depend on the dimension $d$
and on the set $\Lambda$ in the statement of Theorem~\ref{thm:BMadd}.
$\alpha\asymp \beta $ means that $\alpha,\beta$ are positive quantities
whose ratio is bounded above and below by finite, positive constants
which depend only on $d,\Lambda$.

Three symmetrizations of a set $A\subset\reals^d$ will be employed.
$A^*$ denotes the Schwarz symmetrization; that is, $\set{x': (x',x_d)\in A}$ 
is symmetrized as a subset of $\reals^{d-1}$ for each $x_d\in\reals^1$.
$A^\star$ denotes the Steiner symmetrization; that is, $\set{x_d: (x',x_d)\in A}$ 
is symmetrized as a subset of $\reals^{1}$ for each $x'\in\reals^{d-1}$.
$A^*$ and $A^\star$ are defined more precisely below.
Lastly, $A^\doublesymm = (A^\star)^*$ is the set obtained by applying first Steiner symmetrization to $A$,
then Schwarz symmetrization to $A^\star$.

A sumset $A+B$ is by definition empty if either summand is empty.
The Brunn-Minkowski inequality in the form $|A+B|^{1/d}\ge |A|^{1/d}+|B|^{1/d}$
does not hold true if one summand is empty and the other has positive measure. 
\end{notation}

\section{Preliminaries on symmetrization}
To any measurable set $A\subset\reals^d$ of finite measure are associated
the following three types of symmetrizations $A^*,A^\star,A^\doublesymm\subset\reals^d$.
\begin{definition} 
(i) The Schwarz symmetrization $A^*\subset\reals^d$ of $A$ is defined to satisfy:
For each $t\in\reals^1$, 
if
$\set{x\in\reals^{d-1}: (x,t)\in A}$
is measurable with positive $d-1$--dimensional Lebesgue measure then
$\set{x\in\reals^{d-1}: (x,t)\in A^*}$
is an open disk centered at $0\in\reals^{d-1}$, whose $(d-1)$--dimensional measure
is equal to the $(d-1)$--dimensional measure of $\set{x\in\reals^{d-1}: (x,t)\in A}$.
Otherwise 
$\set{x\in\reals^{d-1}: (x,t)\in A^*}$ is empty.

\noindent (ii) The Steiner symmetrization $A^\star\subset\reals^d$ of $A$ is defined to satisfy:
For $x\in\reals^{d-1}$, $A^\star_x=\emptyset$ if $|A_x|=0$. If $A_x$ is Lebesgue measurable and $|A_x|>0$,
then $A^\star_x$ equals the open interval of length equal to $|A_x|$, centered at $0\in\reals^1$.
Otherwise $A^\star_x=\emptyset$.

\noindent (iii) $A^\doublesymm = (A^\star)^*$.
\end{definition}

In particular, $|A^*|=|A^\star|=|A^\doublesymm|=|A|$.
Sumsets of these symmetrizations satisfy useful inequalities.
\begin{lemma} Let $d\ge 2$.
For any Borel sets $A,B\subset\reals^d$ with finite Lebesgue measures,
\begin{align*} &|A^*+B^*|\le |A+B|
\\ &|A^\star +B^\star |\le |A+B|
\\ & |A^\doublesymm+B^\doublesymm|\le |A+B|.  
\end{align*} \end{lemma}

\begin{proof} 
The third conclusion is an immediate consequence of the first two.
Consider the second inequality. For any $x\in\reals^{d-1}$,
and for any $x',x''\in\reals^{d-1}$,
\begin{equation} \text{If $x'+x''=x$ then
$A_{x'}+B_{x''}\subset (A+B)_x$.}\end{equation} 

Let $x',x''\in\reals^{d-1}$, and consider $x=x'+x''$.
$A^\star_{x'},B^\star_{x''}$ are subintervals of $\reals^1$ centered at $0$, and 
$(A^\star+B^\star)_x$ contains the sum of these two intervals,
which is itself an interval centered at $0$.
Moreover,
\begin{equation} (A^\star+B^\star)_x  = \bigcup_{x'+x''=x}\  (A^\star_{x'}+B^\star_{x''}), \end{equation}
where the union is taken only over pairs $(x',x'')$ for which both sets $A_{x'},B_{x''}$ are nonempty.
Now $A^\star_{x'}+B^\star_{x''}$ is an interval centered at $0$,
whose measure equals $|A^\star_{x'}|+|B^\star_{x''}|$. Therefore
\begin{equation} |(A^\star+B^\star)_x|  
= \sup_{x'+x''=x}\  |A^\star_{x'}|+|B^\star_{x''}|
= \sup_{x'+x''=x}\  |A_{x'}|+|B_{x''}| \end{equation}
By the one-dimensional Brunn-Minkowski inequality,
\begin{equation} |A_{x'}|+|B_{x''}| \le |A_{x'}+B_{x''}|.  \end{equation}
Therefore
\begin{equation} |(A^\star+B^\star)_x|  \le  \sup_{x'+x''=x}\  |A_{x'}+B_{x''}|.  \end{equation}
Since $(A+B)_x\supset A_{x'}+B_{x''}$, we find that
\begin{equation} |(A^\star+B^\star)_x|  \le  (A+B)_x \end{equation}
for all $x\in\reals^{d-1}$.
Therefore by Fubini's theorem, $|A^\star+B^\star| \le  |A+B|$.

The first inequality is proved in the same way.
\end{proof}

To any Borel set $A\subset\reals^d$ is associated the function $\reals^{d-1}\owns x\mapsto |A_x|$.
Superlevel sets of this function for $A^\doublesymm$, and 
superlevel sets of the corresponding function for $A$, satisfy the following relationship.
\begin{lemma}
Let $A\subset\reals^d$ be a measurable set with finite, positive Lebesgue measure.
Then for almost every $\lambda>0$,
\begin{equation} |\set{x\in\reals^{d-1}: |A^\doublesymm_x|>\lambda}| 
= |\set{x\in\reals^{d-1}: |A_x|>\lambda}|.  \end{equation}
\end{lemma}

\begin{proof} 
We will write $A^\doublesymm_x$ as shorthand for $(A^\doublesymm)_x$.
For $s\in\reals^1$ consider $F(s) = |\set{x: (x,s)\in A^\doublesymm}|$.
This equals
$|{y: (y,s)\in A^\star}|$, because $A^\doublesymm  =(A^\star)^*$.
Since $A^\star=\set{(x,s)\in\reals^{d-1}\times\reals^1: |s|<\tfrac12|A_x|}$,
$F$ is an even function, which is nonincreasing on $[0,\infty)$.
Therefore up to a null set, 
\[A^\doublesymm=\set{(x,s)\in\reals^{d-1}\times\reals^1: |s|<\tfrac12 |A^\doublesymm_x|}\]
and for almost every $\lambda$,
\[ \set{x: |A^\doublesymm_x|>\lambda} = \set{x: (x,\tfrac12\lambda)\in A^\doublesymm}.  \]
On the other hand,
\[A^\star=\set{(x,s)\in\reals^{d-1}\times\reals^1: |s|<\tfrac12|A_x|},\]
so for almost every $\lambda$,
\[ \set{x: |A^\star_x|>\lambda} = \set{x: (x,\tfrac12\lambda)\in A^\star}.  \]
By definition of Schwarz symmetrization, the slices $\set{x: (x,r)\in A^\doublesymm}$
and $\set{x: (x,r)\in A^\star}$ have equal Lebesgue measures for almost every $r$.
\end{proof}

\begin{corollary} \label{cor:equalprojectionmeasures}
Let $\pi:\reals^d\to\reals^{d-1}$ be the projection $\pi(x',x_d)=x'$.
If $A\subset\reals^d$ is a nonempty open set then
\begin{equation} |\pi(A)| = |\pi(A^\doublesymm)|.  \end{equation}
\end{corollary}

This follows because for any open set $A$, $\pi(A) = \set{x\in\reals^{d-1}: |A_x|>0}$.

\section{On projections} 
Let $d\ge 2$.
Continue to denote by $\pi:\reals^d\to\reals^{d-1}$ the mapping $\pi(x',x_d)=x'$.
\begin{lemma} Let $d\ge 2$.
For any nonempty Borel sets $A,B\subset\reals^d$  and any $t\in(0,1)$,
\begin{equation} \sup_{x\in\reals^{d-1}}|A_x| \cdot |\pi(B)| \le t^{-1}(1-t)^{-(d-1)}|tA+(1-t)B|.
\end{equation} \end{lemma}

\begin{proof}
For any $x\in\reals^{d-1}$,
$tA+(1-t)B$ contains the set of all $(tx+(1-t)y,s)$ such that $y\in\pi(B)$ and $s\in tA_x+(1-t)B_y$.
Now $|tA_x+(1-t)B_y|\ge t|A_x|+(1-t)|B_y|\ge t|A_x|$ by the one-dimensional Brunn-Minkowski ineqality.
Therefore 
\[|tA+(1-t)B| \ge t|A_x|(1-t)^{d-1}|\pi(B)|.\]
\end{proof}

The same conclusion holds, with the roles of $A,B$ reversed. 
Now suppose that $|A|,|B|\ge\tfrac12$ and $|tA+(1-t)B|\le 2$.
Since $|\pi(A)|\cdot \sup_x|A_x|\ge |A|$, and $|\pi(B)|\cdot \sup_x|B_x|\ge |B|$, 
it then follows that
\[\frac{\sup_x |A_x|}{\sup_y |B_y|} \in[\gamma,\gamma^{-1}] \]
for some $\gamma\in(0,1]$ which depends only on $d$, and likewise
\[\frac{|\pi(A)|}{|\pi(B)|} \in[\gamma,\gamma^{-1}].\]

\begin{corollary} \label{cor:normalize}
Let $d\ge 2$, and let $\Lambda\subset(0,1)$ be a compact set.
There exists $\gamma<\infty$, depending only on $d,\Lambda$, with the following property.
Let $t\in\Lambda$, and let $A,B\subset\reals^d$ be Borel sets satisfying
$|A|\ge \tfrac12$, $|B|\ge \tfrac12$, and $|tA+(1-t)B|\le 2$.
Then there exists a measure-preserving linear automorphism $\phi$ of $\reals^d$
such that the images $\phi(A),\phi(B)$ satisfy
\begin{align*} \sup_{x\in\reals^{d-1}} |\phi(A)_x|&\le\gamma,
\\ \sup_{x\in\reals^{d-1}} |\phi(B)_x|&\le\gamma,
\\ |\pi(\phi(A))|&\le\gamma,
\\ |\pi(\phi(B))|&\le\gamma.  \end{align*}
\end{corollary}

\begin{proof} It suffices to consider $\phi(x,t) = (rx, r^{-(d-1)}t)$ where $r>0$ is chosen so that
$r^{d-1}|\pi(A)|=1$.  \end{proof}

\begin{definition}
Let $d\ge 2$ and $1\le \gamma<\infty$.  Any Borel set $A\subset\reals^d$ which satisfies 
\begin{equation} \left\{ \begin{aligned}
|A|&\ge\gamma^{-1},
\\ \sup_{x\in\reals^{d-1}} |A_x|&\le\gamma
\\ |\pi(A)| &\le\gamma
\end{aligned} \right.  \end{equation}
is said to be $\gamma$--normalized.
\end{definition}

These conditions imply also that $|A|$ is bounded above by a finite quantity depending only on $d,\gamma$.
In these terms, Corollary~\ref{cor:normalize} can be restated as follows.
\begin{corollary} \label{cor:normalize2}
Let $d\ge 2$, and let $\Lambda\subset(0,1)$ be a compact set.
There exists $\gamma<\infty$, depending only on $d,\Lambda$, with the following property.
For any $t\in\Lambda$, and for any pair of Borel sets $A,B\subset\reals^d$ sets satisfying
$|A|\ge \tfrac12$, $|B|\ge \tfrac12$, and $|tA+(1-t)B|\le 2$,
there exists a measure-preserving invertible linear automorphism $\phi$ of $\reals^d$
such that the sets $\phi(A),\phi(B)$  are $\gamma$--normalized.
\end{corollary}

Recall the notation $\scripta(s)=\set{x\in\reals^{d-1}: |A_x|>s}$, for $A\subset\reals^d$.
$\scripta(s)\subset\pi(A)$, and if $s,t>0$ then $\scripta(s+t)\subset\scripta(s)$. 
\begin{lemma} \label{lemma:leveladjustment}
For any Borel set $A\subset\reals^d$, For any $r,t\in(0,\infty)$,
\begin{equation} \big|\set{s\in\reals^+: |\scripta(s)|-|\scripta(s+t)|>r} \big| \le tr^{-1}|\pi(A)|.
\end{equation} \end{lemma}

\begin{proof}
\begin{equation} \int_0^\infty  \Big(|\scripta(s)|-|\scripta(s+t)| \Big)\,ds
= \int_0^t |\scripta(s)|\,ds\le \int_0^t|\pi(A)|\,dx.  \end{equation}
The conclusion follows from this bound by Chebyshev's inequality.  \end{proof}

\section{Properties of special sets and their sums}

\begin{lemma} \label{lemma:equalize1}
Let $d\ge 2$.  Let $0<c<C<\infty$ and $R<\infty$.
Let $\scriptf$ be the set of all Borel measurable subsets $A\subset B(0,R)\subset \reals^d$ 
which satisfy $c\le |A|\le C$. Then $\set{\one_{A^\doublesymm}: A\in\scriptf}$ is a precompact subset
of $\lt(\reals^d)$.
\end{lemma}

\begin{proof}
Fix any exponent $s\in (0,\tfrac12)$.
It is straightforward to show that if $f=\one_{A^\doublesymm}$ with $A\in\scriptf$,
then $\norm{f}_{H^s}$ is finite, and is majorized by a finite constant which
depends only on $d,c,C,R,s$. The proof uses for instance the fact that if
$B\subset\reals^{d-1}$ is any ball of finite radius, then $|\widehat{\one_B}(\xi)|\le
C'(1+|\xi|)^{-(d+1)/2}$ for all $\xi\in\reals^{d-1}$,
where $C'$ is a finite constant provided that the radius of $B$ is bounded above.

Whenever $A$ is contained in a ball $B(0,R)$ centered at the origin, 
Schwarz and Steiner symmetrization produce
sets $A^*,A^\star$ which are subsets of the same ball $B(0,R)$.
Therefore likewise $A^\doublesymm\subset B(0,R)$, so Rellich's lemma guarantees precompactness.
\end{proof}

\begin{lemma}
Let $d\ge 2$, let $\Lambda\subset(0,1)$ be compact, and let $\gamma<\infty$.
Let $(A_\nu,B_\nu)$ be a sequence of pairs of Borel measurable subsets of $\reals^d$,
and let $t_\nu$ be a  convergent sequence of elements of $\Lambda$.
Suppose that $|A_\nu|\to 1$, $|B_\nu|\to 1$, and $|t_\nu A_\nu+(1-t_\nu)B_\nu|\to 1$ as $\nu\to\infty$. 
Suppose further that $A_\nu,B_\nu$ are $\gamma$--normalized for all indices $\nu$.

Suppose finally that both sequences $\one_{A_\nu^\doublesymm},\one_{B_\nu^\doublesymm}$ are convergent
in $\lt(\reals^d)$. Then there exists a $\gamma$--normalized convex set $\scriptc\subset\reals^d$
satisfying $\scriptc=\scriptc^\doublesymm$ such that
\begin{equation} \lim_{\nu\to\infty} |A_\nu^\doublesymm\bigtriangleup \scriptc|
=\lim_{\nu\to\infty} |B_\nu^\doublesymm\bigtriangleup \scriptc|=0.  \end{equation}
\end{lemma}

\begin{proof}
\[
|A_\nu|^t|B_\nu|^{1-t}
=|A_\nu^\doublesymm|^t|B_\nu^\doublesymm|^{1-t} 
\le|tA_\nu^\doublesymm+(1-t)B_\nu^\doublesymm|
\le |tA_\nu+(1-t)B_\nu|\to 1\]
as $\nu\to\infty$, so 
$|tA_\nu^\doublesymm+(1-t)B_\nu^\doublesymm|\to 1$.

Let $t\in(0,1)$.
It is shown in Lemma~14.1 of \cite{christbmtwo} 
that for any sequence of pairs of sets $(\tilde A_\nu,\tilde B_\nu)$ which satisfy all of the conditions 
\begin{equation*}
\left\{
\begin{aligned}
&|\tilde A_\nu|\to 1 \ \text{ and } \  |\tilde B_\nu|\to 1 
\ \text{ as } \ \nu\to\infty
\\
&\text{The sequences}\ (\one_{\tilde A_\nu}) \ \text{ and } \  (\one_{\tilde B_\nu}) 
\ \text{ converge in $\lt(\reals^d)$ as } \ \nu\to\infty
\\
&|t\tilde A_\nu+(1-t)\tilde B_\nu|\to 1
\ \text{ as } \ \nu\to\infty,
\end{aligned} \right.
\end{equation*}
there exist limiting sets $A_\infty,B_\infty$ such that 
$|\tilde A_\nu \bigtriangleup A_\infty|\to 0$,
$|\tilde B_\nu \bigtriangleup B_\infty|\to 0$, $|A_\infty|=|B_\infty|=1$,
and $|tA_\infty+(1-t)B_\infty|=1$.
That same proof demonstrates that under the more general hypotheses of this lemma,
the same conclusion holds with $t=\lim_{\nu\to\infty} t_\nu$.

Apply this with $(\tilde A_\nu,\tilde B_\nu)=(A^\doublesymm_\nu,B^\doublesymm_\nu)$.
Since $A_\nu^\doublesymm = (A_\nu^\doublesymm)^\doublesymm$, $A_\infty=A_\infty^\doublesymm$. 
Likewise $B_\infty=B_\infty^\doublesymm$.

By the known characterization of equality in the Brunn-Minkowski inequality \cite{HO},
there exist an open convex set $\scriptc\subset\reals^d$ and a vector $z\in\reals^d$ 
such that $A_\infty\bigtriangleup\scriptc$ and
and $B_\infty\bigtriangleup(\scriptc+z)$ are null sets.
The relation $|A_\infty\bigtriangleup A_\infty^\doublesymm|=0$ then forces 
$|\scriptc^\doublesymm\bigtriangleup\scriptc|=0$,
and the relation $|B_\infty\bigtriangleup B_\infty^\doublesymm|=0$ likewise forces 
$|(\scriptc+z)^\doublesymm\bigtriangleup (\scriptc+z)|=0$, which in turn forces $z=0$.
\end{proof}

It follows quite easily that  the conclusion holds with $z$ replaced by $0$,
but we will not need that fact.

\begin{lemma} \label{lemma:convexquadratictail}
Let $d\ge 1$ and $C_0<\infty$.
There exists $C>0$ with the following property.
Consider any convex set $\scriptc\subset\reals^d$ satisfying $|\scriptc|=1$
such that $\scriptc=\scriptc^\doublesymm$, and $|\pi(\scriptc)|\le C_0$.
Let $\scriptc(\lambda)=\set{x\in\reals^{d-1}:|\scriptc_x|>\lambda}$.
Then  for every $\eps>0$,
\begin{equation} |\scriptc\setminus\pi^{-1}(\scriptc(\eps))|\le C\eps^2.  \end{equation}
\end{lemma}

\begin{proof}
$\scriptc$ takes the form $\scriptc=\set{(x,t)\in\reals^{d-1}\times\reals^1: |t|<h(|x|)}$
where $h:[0,\infty) \to[0,\infty)$ is a nonincreasing function which is convex
on the interval $\set{s\in[0,\infty): h(x)>0}$.
Since $|\pi(\scriptc)|\le C_0$, this interval is of the form $[0,\sigma)$ where $\sigma\le\sigma_0(C_0,d)<\infty$.
Since $\scriptc$ is convex but has finite measure, $h(0)$ must be finite.
On the other hand, $1=|\scriptc|\le c_d \sigma^d h(0)$, so $h(0)$ is bounded below by a positive constant
$h_0(C_0,d)$ which depends only on $d,C_0$. Since $x\mapsto h(|x|)$ is convex, $h$ is bounded below 
on the interval $[0,\sigma]$ by the affine function $H$ which satisfies $H(0)=h(0)$ and $H(\sigma)=0$.
Therefore 
\begin{equation*} \scriptc\setminus\pi^{-1}(\scriptc(\eps)) =\set{(x,t)\in \scriptc:  h(|x|)<\tfrac12 \eps}
\subset \set{(x,t)\in\reals^{d-1}\times[-\tfrac12 \eps,\tfrac12 \eps]:  
H(|x|)<\tfrac12 \eps}.  \end{equation*}
Since $|dH/ds|$ is bounded below by a positive constant which depends only on $C_0,d$,
\[ |\set{x \in\reals^{d-1}:  H(|x|)<\tfrac12\eps}| \le C\eps.  \]
\end{proof}

\section{On the function $x\mapsto |A_x|$ for near extremizers}

Let $d\ge 2$. Let $\Lambda$ be a compact subset of $(0,1)$, and let $t\in\Lambda$.
Let $A,B\subset\reals^d$ be Borel sets which satisfy 
\begin{equation} \label{hypos} \left\{
\begin{aligned} & \big|\,|A|-1\,\big| < \delta\ \text{ and } \big|\,|B|-1\,\big| < \delta 
\\ & |tA+(1-t)B|<1+\delta.  \end{aligned} \right. \end{equation}
These will be our standing hypotheses for the remainder of the proof.
As was shown in \cite{christbmtwo}, in order to establish Theorem~\ref{thm:BMadd},
it suffices to show that under these hypotheses, there exist a convex set $\scriptc\subset\reals^d$
and a vector $v\in\reals^d$ such that $A\subset\scriptc$, $B+v\subset\scriptc$,
and $|\scriptc|<1+o_\delta(1)$, where the quantity denoted by $o_\delta(1)$ depends only
on $\Lambda,d,\delta$, and tends to zero as $\delta\to 0$ so long as $\Lambda,d$ remain fixed. 

For any measure-preserving linear automorphism $\phi$ of $\reals^d$,
$|\phi(A)|=|A|$, $|\phi(B)|=|B|$, and \[|t\phi(A)+(1-t)\phi(B)| = |\phi(tA+(1-t)B| = |tA+(1-t)B|.\]
Thus if a pair of sets $(A,B)$ satisfies the hypotheses \eqref{hypos}, so does $(\phi(A),\phi(B))$.
Corollary~\ref{cor:normalize2}  asserts that if $A,B$ satisfy these hypotheses then
there exist $\phi$ such that $\phi(A),\phi(B)$ are $\gamma$--normalized,
where $\gamma<\infty$ depends only on $d,\Lambda$.
We will write simply ``normalized'' to mean $\gamma$--normalized with some such constant $\gamma$.
We therefore may, and will, assume henceforth that $A,B$ are normalized.

\begin{lemma} \label{lemma:squareerror}
For any dimension $d\ge 2$,
there exist a constant $C<\infty$ and a function $\eps(\delta)$
satisfying $\lim_{\delta\to 0}\eps(\delta)=0$ with the following properties.
Let $\delta>0$ be sufficiently small.  For any normalized Borel sets
$A,B\subset\reals^d$ satisfying \eqref{hypos}, for all $\eta\ge\eps(\delta)$,  
\begin{equation} \label{quadraticinverseprojectionbounds} \begin{aligned}
& |A\setminus\pi^{-1}(\scripta(\eta))| \le C\eta^2
\\ & |B\setminus\pi^{-1}(\scriptb(\eta))| \le C\eta^2.
\end{aligned} \end{equation}
Moreover, for all $s>0$,
\begin{equation} \label{nearlysamedistfns} \begin{aligned} 
&|\scripta(s)|\le |\scriptb(s-\eps(\delta))| + \eps(\delta) 
\\ &|\scriptb(s)|\le |\scripta(s-\eps(\delta))| + \eps(\delta).
\end{aligned} \end{equation}
\end{lemma}

\begin{proof}
Recall that $\scripta(s)=\set{x\in\reals^{d-1}: |A_x|>s}$ and $|\scriptb(s)|=|\set{x: |B_x|>s}|$.
These are the distribution functions
of the $L^1(\reals^{d-1})$ functions $\reals^{d-1}\owns x\mapsto |A_x|$ and $x\mapsto |B_x|$.
The functions $x\mapsto |A_x|$ and $x\mapsto |A^\doublesymm_x|$ have 
identical distribution functions, and likewise for the pair $B, B^\doublesymm$.
The conclusions of this lemma are all statements about these distribution functions alone. 
Therefore it suffices to prove the lemma with $A,B$ replaced by $A^\doublesymm,B^\doublesymm$, respectively;
that is, under the assumption that $A=A^\doublesymm$ and $B=B^\doublesymm$.

We have shown in Lemma~\ref{lemma:convexquadratictail}  that, under the assumption
that $A=A^\doublesymm$ and $B=B^\doublesymm$, there exists a 
convex set $\scriptc\subset\reals^d$ such that $|A\bigtriangleup \scriptc|+|B\bigtriangleup\scriptc|=o_\delta(1)$.
Thus
\begin{equation} \label{smallLonedifferences}
\begin{aligned} &
\int_{\reals^{d-1}} \big|\, |A_x|-|\scriptc_x|   \,\big|\,dx 
=|A\bigtriangleup\scriptc| = o_\delta(1)
\\ & \int_{\reals^{d-1}} \big|\, |B_x|-|\scriptc_x|   \,\big|\,dx 
=|B\bigtriangleup\scriptc| = o_\delta(1).  \end{aligned}
\end{equation}
Therefore
\[\int_{\reals^{d-1}} \big|\, |A_x|-|B_x|   \,\big|\,dx \le \rho(\delta)
\ \text{ where $\rho(\delta)\to 0$ as $\delta\to 0$.} \]
Therefore by Chebyshev's inequality, there exists a set $E\subset\reals^{d-1}$ satisfying 
$|E|\le\rho(\delta)^{1/2}$  such that
$\big|\, |A_x|-|B_x| \,\big|\le \rho(\delta)^{1/2}$
for all $x\in\reals^{d-1}\setminus E$. 
This directly implies \eqref{nearlysamedistfns} with $\eps(\delta)=\rho(\delta)^{1/2}$.

By the same reasoning, there exists a set $\tilde E$
satisfying $|\tilde E|\le\rho(\delta)^{1/2}$ such that
$\big|\, |A_x|-|\scriptc_x| \,\big|\le \rho(\delta)^{1/2}$
for all $x\in\reals^{d-1}\setminus\tilde E$. 
We have shown in Lemma~\ref{lemma:convexquadratictail} that 
$ |\scriptc\setminus\pi^{-1}(\scriptc(\eta))| \le C\eta^2$
for all $\eta>0$, where $\scriptc(s)=\set{x\in\reals^{d-1}: |\scriptc_x|>s}$. 
Combining these statements gives \eqref{quadraticinverseprojectionbounds} for $A$.
The conclusion for $B$ follows in the same way.
\end{proof}

\section{Horizontal structure}

Let $A,B,\Lambda,d,t$ satisfy our hypotheses \eqref{hypos} and be normalized.
Form their sumset $S=tA+(1-t)B$.  Consider the superlevel sets
\begin{align*} \scripta(s)=\set{x: |A_x|>s}, \qquad  \scriptb(s) =\set{x: |B_x|>s},
\qquad  \scripts(s)=\set{x: |S_x|>s}.  \end{align*}

If $A\subset\reals^d$ is a convex set, then its projection $\pi(A)$ is likewise convex.
Our purpose here is to show that if a pair $(A,B)$ realizes near equality
in the Brunn-Minkowski inequality, then 
sets closely related to $\pi(A)$ are approximately convex. The particular
sets in question are the superlevel sets $\scripta(s)$
for small $s>0$. Such a conclusion is however not quantitatively affine-invariant,
so it is necessary to reduce matters to normalized sets throughout the discussion.

For any $x',x''\in\reals^{d-1}$,
\begin{equation} S_{tx'+(1-t)x''} \supset tA_{x'}+(1-t)B_{x''}\end{equation}
and therefore by the Brunn-Minkowski inequality for $\reals^1$,
if $|A_{x'}|>\lambda$ and $|B_{x''}|>\mu$ then 
$|S_{tx'+(1-t)x''}|\ge t\lambda+(1-t)\mu$.
Therefore
\begin{equation}
\set{x: |S_x|>t\lambda+(1-t)\mu}\supset t\set{x': |A_{x'}|>\lambda}+(1-t)\set{x'': |B_{x''}|>\mu}. 
\end{equation}
By the Brunn-Minkowski inequality for $\reals^{d-1}$,
\begin{equation}
|\scripts(t\lambda+(1-t)\mu)| \ge |\scripta({\lambda})^t|\scriptb(\mu)|^{1-t}.
\end{equation}

Let $\delta\mapsto\eps(\delta)$ be the function which appears in the key Lemma~\ref{lemma:squareerror}.
To simplify notation, we write $\eps$ for $\eps(\delta)$ in the following discussion.
Let $\sigma $ be the supremum of the set of all $s\ge 0$ such that
both $|\scriptb(s)|>0$ and $|\scripta(s+\eps)|>0$. Then
for any $s\in(0,\sigma)$, 
\[ |\scripts(s+t\eps)| \ge |\scripta(s+\eps)|^t|\scriptb(s)|^{1-t}
\ge |\scriptb(s)|-O(\eps).  \]
The first inequality was noted in the preceding paragraph; the definition of $\sigma$
ensures that $\scripta(s+\eps),\scriptb(s)$ have positive measures, so in particular are nonempty.
The second inequality is a conclusion of Lemma~\ref{lemma:squareerror}.

\begin{lemma}
\[ \int_\sigma^\infty |\scriptb(s)|\,ds = O(\eps).  \]
\end{lemma}

\begin{proof}
Certainly
\[ \int_\sigma^{\sigma+2\eps} |\scriptb(s)|\,ds \le \int_\sigma^{\sigma+2\eps} |\pi(B)|\,ds = O(\eps).  \]
Suppose that there exists $s>\sigma+2\eps$ for which $|\scriptb(s)|>0$.  Then $\sigma$, 
by its definition, must equal the supremum of the set of all $s'$ for which $|\scripta(s')|>0$.
Thus $|\scripta(s-\eps)|=0$ and consequently by Lemma~\ref{lemma:squareerror},
$|\scriptb(s)|<\eps$. By Corollary~\ref{cor:normalize}, there exists $\gamma=O(1)$
such that $|\scriptb(s)|=0$ for all $s>\gamma$.  Therefore
\[ \int_{\sigma+2\eps}^\infty |\scriptb(s)|\,ds = \int_{\sigma+2\eps}^\gamma |\scriptb(s)|\,ds 
\le  \int_{\sigma+2\eps}^\gamma |\pi(B)|\,ds =O(\eps). \]
\end{proof}

For $0<s<\sigma$,
$ |\scripts(s+t\eps)| -|\scripta(s+\eps)|^t|\scriptb(s)|^{1-t}\ge 0$. 
\begin{lemma}
\[ \int_0^{\sigma} 
\Big( |\scripts(s+t\eps)| -|\scripta(s+\eps)|^t|\scriptb(s)|^{1-t} \Big)
\,ds = O(\delta+\eps(\delta)).  \]
\end{lemma}

\begin{proof}
\begin{align*}
\int_0^\sigma \Big( |\scripts(s+t\eps)| -|\scripta(s+\eps)|^t|\scriptb(s)|^{1-t} \Big)\,ds
& \le 
\int_0^\sigma \Big( |\scripts(s+t\eps)| -|\scriptb(s)|+O(\eps) \Big)\,ds
\\ & \le 
|S| -|B| + O(\eps)+ \int_\sigma^\infty |\scriptb(s)|\,ds
\\ & \le 
O(\delta)+O(\eps) 
\end{align*}
by the preceding lemma.
\end{proof}

By replacing $\eps(\delta)$ by $\eps(\delta)+\delta$, we may assume  that $\eps(\delta)\ge\delta$, 
Let $\eta=\eta(\delta)$, $\rho=\rho(\delta)$  be positive quantities which tend to zero as $\delta\to 0$. 
Choose $\rho(\delta)$ to tend to zero sufficiently slowly that 
\begin{equation} \label{rhocondition} 
\rho(\delta)^{-1}\eps(\delta)\to 0 \ \text{ as $\delta\to 0$.} \end{equation}
Choose $\eta(\delta)$ to tend to zero sufficiently slowly that 
\begin{equation} \label{etacondition}
\lim_{\delta\to 0} \frac{\rho(\delta)^{-1}\eps(\delta)+\eps(\delta)^{1/2}}{\eta(\delta)} = 0 
\end{equation}
and so that $\eta(\delta)$ satisfies another condition \eqref{secondetacondition} of the same
type, to be specified below.

Write $\eps=\eps(\delta)$.
Apply Chebyshev's inequality to conclude that 
\[ \Big|\set{s\in[0,\sigma]: 
\Big( |\scripts(s+t\eps) -|\scripta(s+\eps)|^t|\scriptb(s)|^{1-t} \Big) \,ds >\rho }\Big|
=O(\rho^{-1}\eps).  \]
Therefore by the properties \eqref{rhocondition}, \eqref{etacondition} of $\rho,\eta$,
there exists $\bar s\in[\eta,2\eta]$ such that
\begin{equation} |\scripts(\bar s+t\eps)| <|\scripta(\bar s+\eps)|^t|\scriptb(\bar s)|^{1-t} + \rho.
\end{equation}
Therefore
\begin{equation} \label{eq:rhosumbound}
|t\scripta(\bar s+\eps) +(1-t) \scriptb(\bar s)| <|\scripta(\bar s+\eps)|^t|\scriptb(\bar s)|^{1-t} + \rho.
\end{equation}


We will also need to know that the measures of $\scripta(\bar s+\eps)$ and of $\scriptb(\bar s)$
are nearly equal.
For any $s$,
\begin{equation} |\scripta(s-\eps)|\le |\scriptb(s)|+\eps\le|\scripta(s+\eps)|+2\eps \end{equation}
by Lemma~\ref{lemma:squareerror}.
By Lemma~\ref{lemma:leveladjustment} and the fact that $|\pi(A)|=O(1)$,
$|\scripta(s-\eps)|\le |\scripta(s+\eps)|+\eps^{1/2}$
for all $s$ outside of some exceptional set of measure $O(\eps^{1/2})$.
If $\delta$ is sufficiently small then for any such favorable value of $s$, 
\begin{equation*} 
\big|\,|\scripta(\bar s+\eps)|-|\scriptb(\bar s)\,\big| \le 2\eps^{1/2}\ll \eta(\delta)
\end{equation*}
since $\frac{\eta(\delta)}{\eps(\delta)^{1/2}}\to\infty$ as $\delta\to 0$.

Therefore for sufficiently small $\delta$ there exists $\bar s\in[\eta,2\eta]$ such that 
\eqref{eq:rhosumbound} holds and also
\begin{equation} \label{eq:nearlyequalmeasures}
\big|\,|\scripta(\bar s+\eps)|-|\scriptb(\bar s)\,\big| \le 2\eps^{1/2}.
\end{equation}

Fix such a parameter $\bar s$,
and set $\scripta=\scripta(\bar s+\eps)$ and $\scriptb=\scriptb(\bar s)$.
Since $\eta(\delta)\to 0$, 
$|\scriptb|$ and $|\scripta|$ are bounded below by a strictly positive constant
for all sufficiently small $\delta$.
Provided that $\rho=\rho(\delta)\to 0$ as $\delta\to 0$, it follows from 
\eqref{eq:rhosumbound} and \eqref{eq:nearlyequalmeasures}
that for any sufficiently small $\delta$,
the pair of sets $\scripta,\scriptb$ satisfies the hypotheses of Theorem~\ref{thm:BMadd} 
in dimension $d-1$.
Therefore by the induction hypothesis, whenever $\delta$ is sufficiently small,
there exist a convex set $\scriptc\subset\reals^{d-1}$  containing $0$
and independent translations of  $A,B$, which we continue to denote by $A,B$ respectively, such that
\begin{align*}
&\scriptb\subset\scriptc \ \text{ and } \ |\scriptc\setminus\scriptb|=o_{\rho,\eps}(1),
\\ &\scripta\subset\scriptc \ \text{ and } \ |\scriptc\setminus \scripta|=o_{\rho,\eps}(1).
\end{align*}
It follows from what has been shown
that these quantities denoted by $o_{\rho,\eps}(1)$ are majorized by functions of $\rho,\eps$ alone;
this will be important in the sequel.

Now there is the obvious bound \[|A\setminus \pi^{-1}(\scriptc)|\le |A\setminus \pi^{-1}(\scripta)|
\le (\bar s+\eps)|\pi(A)|\lesssim \bar s+\eps = O(\eta(\delta)).\]
However, Lemma~\ref{lemma:squareerror} guarantees the improved bounds
\begin{equation} \label{betterbound}
|A\setminus \pi^{-1}(\scripta)| + |B\setminus\pi^{-1}(\scriptb)| = O(\eta(\delta)^2).  \end{equation}

\section{Vertical structure}
Our next step is to show that for most $x\in\reals^{d-1}$ the fiber sets $A_x$
are nearly convex --- that is, nearly coincide with intervals. This step
relies on the one-dimensional case of our main theorem, which has the following more precise formulation.
\begin{proposition} \label{prop:1Dnearequality} 
Let $U,V\subset\reals^1$ be nonempty Borel sets with finite Lebesgue measures.
If $|U+V|<|U|+|V|+\eps$, and if $\eps<\min(|U|,|V|)$,
then there exists an interval $I\subset\reals^1$ satisfying 
\begin{equation} I\supset U \ \text{ and } \ |I\setminus U|< |U|+\eps.  \end{equation}
\end{proposition}

A delicate point in applications of this result is the necessity to ensure 
validity of the hypothesis $|U+V|-|U|-|V|<\min(|U|,|V|)$. 
In particular, the Brunn-Minkowski inequality can only be applied to fibers $A_x,B_y$
when both are known to be nonempty.  The crucial point in this regard will be 
the appearance of a bound $o(\eta)$ in \eqref{betterbound}.

\begin{lemma} \label{lemma:vertical}
Let $d\ge 2$, and let $\Lambda\subset(0,1)$ be a compact set.
There exists a positive function $\delta\mapsto\eta(\delta)$
satisfying $\lim_{\delta\to 0} \eta(\delta)=0$,  with the following property.
For any $t\in\Lambda$, any sufficiently small $\delta>0$,
and any Borel sets $A,B\subset\reals^d$ satisfying the hypotheses \eqref{hypos},
there exist a measure-preserving linear transformation $L$ of $\reals^d$
and vectors $a,b\in\reals^d$ such that after replacement of $A,B$ by $L(A)+a$, $L(B)+b$ respectively,
there exist a convex set $\scriptc\subset\reals^{d-1}$
and a  measurable subset $\scriptd\subset\scriptc$ such that
\begin{align*}
& |\scriptc|\asymp 1 
\\
&|\scriptc\setminus\scriptd|=o_\delta(1)
\\
&|A\cap \pi^{-1}(\reals^d\setminus \scriptd)| = o_\delta(1)
\\
&|B\cap \pi^{-1}(\reals^d\setminus \scriptd)| = o_\delta(1).
\end{align*}
Moreover,
\begin{equation} \label{turningpoint} \left\{ \begin{aligned}
&\min(|A_x|,|B_x|)\ge\eta(\delta) \text{ for each } x\in \scriptd
\\
&|S_x|-t|A_x|-(1-t)|B_x| \le  \eta(\delta)^{3/2} \text{ for each } x\in\scriptd.
\end{aligned} \right.  \end{equation}
\end{lemma}

Before proving Lemma~\ref{lemma:vertical}, we record a consequence.
For small $\delta$, $\eta^{3/2}\ll\eta$ and therefore
\[|S_t|-t|A_x|-(1-t)|B_x|\ll \min(|A_x|,|B_x|)\ \text{ for each $x\in\scriptd$.}\]
This crucial point, guaranteed in more quantitative form by \eqref{turningpoint},
depends on there being a bound $o_\eta(1)$ in the first two conclusions of Lemma~\ref{lemma:squareerror}.

A direct application of Proposition~\ref{prop:1Dnearequality} gives:
\begin{corollary}
Let $A,B$ satisfy the hypotheses of Lemma~\ref{lemma:vertical}, and let $\scriptd$
satisfy the conclusions of that lemma.
Then for each $x\in\scriptd$ there exist intervals $I_x,J_x\subset\reals^1$
such that $A_x\subset I_x$, $B_x\subset J_x$, 
\begin{equation}
|I_x|<|A_x|+o_\delta(1)
\ \text{ and } \ 
|J_x|<|B_x|+o_\delta(1).
\end{equation}
These intervals can be chosen so that their centers and lengths are measurable functions of $x\in\scriptd$.
\end{corollary}

\begin{proof}[Proof of Lemma~\ref{lemma:vertical}]
For $x\in\scripta\cap\scriptb$, $|S_{x}| \ge t|A_x|+(1-t)|B_x|$
by the one-dimensional Brunn-Minkowski inequality.
\begin{align*}
\int_{\scripta\cap\scriptb} \Big(|S_x|-&t|A_x|-(1-t)|B_x| \Big)\,dx
\\
 &\le |S| -t|A|-(1-t)|B|
 + t\int_{\reals^{d-1}\setminus\scripta}|A_x|\,dx
\\
& \qquad + (1-t) \int_{\reals^{d-1}\setminus\scriptb}|B_x|\,dx
 +O(|\scriptc\setminus (\scripta\cap\scriptb)|)
\\ &= |S| -t|A|-(1-t)|B| +t|A\setminus \pi^{-1}(\scripta)| +(1-t)|B\setminus \pi^{-1}(\scriptb)|
+ o_{\rho,\eps}(1).
\end{align*}
Thus
\begin{equation} \label{keyintegralbound}
\int_{\scripta\cap\scriptb} \Big(|S_x|-t|A_x|-(1-t)|B_x| \Big)\,dx
\le  \delta + O(\eta^2) + o_{\rho,\eps}(1) \end{equation}
by \eqref{betterbound}.

Define  $\scriptd\subset\scripta\cap\scriptb$ to be the set
\begin{equation}
\scriptd = \set{x\in\scripta\cap\scriptb:  |S_x|-t|A_x|-(1-t)|B_x| \le \eta^{3/2}}.
\end{equation}
By Chebyshev's inequality and the preceding bound for the integral of $|S_x|-t|A_x|-(1-t)|B_x|$,
\begin{equation}
|(\scripta\cap\scriptb)\setminus \scriptd| \le O(\eta^{-3/2}\delta) + O(\eta^{1/2}) + \eta^{-3/2}o_{\rho,\eps}(1).
\end{equation}

The three quantities $\eps,\rho,\eta$ depend on $\delta$ and are required to tend to zero
as $\delta$ tends to zero.
By its construction,
$\lim_{\delta\to 0} \eps(\delta)\to 0$ as $\delta\to 0$. $\rho(\delta)$ is
defined in terms of $\eps(\delta)$ and was chosen above to tend to zero
sufficiently slowly to satisfy \eqref{rhocondition}. $\eta(\delta)$ depends
on the choices of $\eps(\delta),\rho(\delta)$, and 
and has already been required to tend to zero sufficiently slowly to satisfy \eqref{etacondition}.
We now impose on $\eta(\delta)$ the additional requirement that 
\begin{equation} \label{secondetacondition}
\eta(\delta)^{-3/2}\delta +  \eta(\delta)^{-3/2}o_{\rho,\eps}(1)  \to 0 \text{ as } \delta\to 0
\end{equation}
where $o_{\rho,\eps}(1)$ has the same value as in the right-hand side of \eqref{keyintegralbound}.
Two conditions \eqref{etacondition} and \eqref{secondetacondition} are now required of $\eta(\delta)$.
Both require merely that $\eta(\delta)$ tend to zero sufficiently slowly as a function of quantities defined
without reference to $\eta(\delta)$,
namely $\delta,\eps(\delta),\rho(\delta)$ and the quantity denoted by $o_{\rho,\eps}(1)$ above. Therefore there
does exist such a function $\eta$.
\end{proof}

\section{Affine structure} 
Let $A,B,\scriptc,\scriptd,I_x,J_x$ be as above.
At this stage, we know that nearly all sets $A_x$ nearly coincide with intervals $I_x$,
and we know that the set of all $x\in\reals^{d-1}$ for which $|A_x|$
is nonnegligible, nearly coincides with a convex set. But we have no control whatsoever on the 
manner in which the intervals $I_x$ depend on $x$.
The next step is to control this dependence. 
The information in hand is invariant under skew shifts $(x',x_d)\mapsto (x', x_d + v\cdot x')$,
so any conclusions must allow for this same invariance.
The analysis here will be very nearly identical to the corresponding step
in the two-dimensional analysis already given in  \cite{christbmtwo}.
The main ingredient was established in \cite{christyoungest}.

For $x\in\scriptd$ define $\varphi(x),\psi(x)$ to be the centers of $I_x,J_x$ respectively.
\begin{lemma}
Let $d\ge 1$. There exists $\gamma<\infty$ with the following property.
If $\delta$ is sufficiently small then there exist a subset $\scriptd^\sharp\subset\scriptd$, 
an affine function $\phi:\reals^d\to\reals$, and a scalar $v\in\reals^1$, such that 
\begin{equation}
|\scriptd\setminus\scriptd^\sharp|=o_\delta(1)
\end{equation}
and
\begin{equation}
|\varphi(x)-\phi(x)|\le \gamma
\ \text{ and } \ 
|\psi(x)-\phi(x)-v|\le \gamma
\ \text{ for all } \ x\in\scriptd^\sharp.
\end{equation}
\end{lemma}

\begin{proof}
Suppose that $x_1,x_2,y_1,y_2$ all belong to $\scriptd$ and $tx_1+(1-t)y_1 = tx_2+(1-t)y_2=z\in\reals^{d-1}$.
The intervals $I_{x_j}$, $J_{x_j}$ have lengths $O(1)$. 
$S_z$ contains $tA_{x_j}+(1-t)B_{y_j}$ for $j=1,2$. These are subsets of the intervals $tI_{x_j}+(1-t)J_{y_j}$
respectively, which have uniformly bounded lengths and centers $c_j=t\varphi(x_j)+(1-t)\psi(y_j)$.
There exists an absolute constant $\gamma\in\reals^+$ such that if $|c_1-c_2|>\gamma$,
then \[(tI_{x_1}+(1-t)J_{y_1}) \cap (tI_{x_2}+(1-t)J_{y_2})=\emptyset.\]
In that case,
\begin{equation} \label{eq:twodisjointintervals}
|S_z|\ge t|I_{x_1}|+(1-t)|J_{y_1}| + t|I_{x_2}| + (1-t)|J_{y_2}|.
\end{equation}

We define a set $\scriptd'\subset\scriptd$ by defining its complement $\scriptd\setminus \scriptd'$ to be
the set of all $z\in \scriptd$ for which there exist $x',x''\in\scriptd$
satisfying $tx'+(1-t)x''=x$ such that 
\[\big|\big(|t\varphi(x')+(1-t)\psi(x'')\big) -\big(|t\varphi(x)+(1-t)\psi(x)\big)\big|>\gamma.\]

For $x\in\scriptd\setminus\scriptd'$, there exist $x',x''\in\scriptd$ such that 
\[x=tx'+(1-t)x'' \ \text{ and } \ 
(tI_{x}+(1-t)J_{x}) \cap (tI_{x'}+(1-t)J_{x''})=\emptyset.\]
Since  $x',x''\in\scriptd\subset\scripta\cap\scriptb=\scripta(\bar s+\eps)\cap\scriptb(\bar s)$
and $\bar s\ge\eta$, $|A_{x'}|\ge\eta$ and $|B_{x''}|\ge\eta$. Consequently
\begin{align*}
|tI_{x'}+(1-t)J_{x''}| &\ge t|I_{x'}|+(1-t)|J_{x''}|
\\
&\ge t|A_{x'}|+(1-t)|B_{x''}|
\\
&\ge t\eta +(1-t)\eta
\\
&=\eta,
\end{align*}
and the disjointness of these two intervals implies the lower bound
\[ |S_x|\ge t|A_x|+(1-t)|B_x|+\eta.  \]
Integrating gives
\[
\int_{\scriptd\setminus\scriptd'} \big( |S_x|- t|A_x|-(1-t)|B_x|\big)\,dx \ge \eta|\scriptd\setminus\scriptd'|.
\]
Since $\scriptd\subset\scripta\cap\scriptb$, a favorable bound for the integral has already been established:
\begin{align*}
|\scriptd\setminus\scriptd'| 
&\le \eta^{-1} \int_{\scripta\cap\scriptb} \big( |S_x|- t|A_x|-(1-t)|B_x|\big)\,dx
\\
&\le \eta^{-1}(O(\delta)+O(\eta^2)+o_\rho(1))
\\
&= O(\eta^{-1}\delta)+O(\eta)+\eta^{-1}o_\rho(1).
\end{align*}
$\eta=\eta(\delta)$ has been chosen so that the right-hand side in this inequality is $o_\delta(1)$.

The next step is to produce the required affine function $\phi$. 
This is done in the same way as the corresponding step in \cite{christbmtwo},
with the small additional complication that $\scriptc$ is here a general convex set,
while in \cite{christbmtwo} it was an interval.

For $x\in\scriptd'$, define $\zeta(x)\in\reals^1$ by $\zeta(x) = t\varphi(x)+(1-t)\psi(x)$.
Then $|\scriptc\setminus\scriptd'|=o_\delta(1)$,
and the set of all $(x',x'')\in \scriptd'\times\scriptd'$ such that $tx'+(1-t)x''\notin\scriptd'$
has measure $o_\delta(1)$. Finally,
for all $x',x''\in\scriptd'$ such that $tx'+(1-t)x''\in\scriptd'$,
\[ |t\varphi(x')+(1-t)\psi(x'')-\zeta(tx'+(1-t)x'')|\le \gamma.  \]
Since $\scriptc$ is convex, there exists an ellipsoid contained in $\scriptc$, whose Lebesgue
measure is comparable to the measure of $\scriptc$.
Make a measure-preserving affine change of variables, which transforms that ellipsoid into a ball,
denoted by $\scripte$.
By Lemma~6.6 of \cite{christyoungest},
there exist an affine mapping $\phi:\reals^{d-1}\to\reals$ and a scalar $v\in\reals^{1}$ such that
\begin{align*} &|\varphi(x)-\phi(x)|\le C\gamma \ \text{ for all $x\in\scriptd''$}
\\ &|\psi(x)-\phi(x)-v|\le C\gamma \ \text{ for all $x\in\scriptd''$} \end{align*}
where
\begin{equation*} |\scripte\setminus\scriptd''|=o_\delta(1).  \end{equation*}
This is the desired conclusion, except that it has been proved only for a large subset
of $\scripte$, rather than a large subset of $\scriptd$.

We may translate so that $\scripte=B(0,r)$, the ball centered at $0\in\reals^{d-1}$
of radius $r$, where $r\asymp 1$. Then $\scriptc\subset B(0,Cr)$ where $C\in\reals^+$
depends only on the dimension $d$.
By composing with the measure-preserving affine automorphism of $\reals^d=\reals^{d-1}\times\reals^1$
defined by $(y;u)\mapsto (y;u-\phi(x))$, and by translating $B$ by $(0;v)$,
we may reduce to the case in which $\phi$ is the identity mapping, and $v=0$.
Extend $\varphi,\psi,\zeta$ in an arbitrary way to measurable functions
from $\scriptc\subset\reals^{d-1}$ to $\reals^1$.
Then
$|t\varphi(x')+(1-t)\psi(x'')-\zeta(tx'+(1-t)x'')|\le \gamma$ for all $(x',x'')\in\scriptc^2$
except for a set of Lebesgue measure $o_\delta(1)$,
and $\varphi,\psi,\zeta=O(1)$ at every point of on $B(0,r)$, with the exception of a set of points having
measure $o_\delta(1)$.

From this it follows that $\varphi,\psi,\zeta=O(1)$ at each point of $\scriptc$, with the exception
of a set of points $\scriptd^\natural$ having measure $o_\delta(1)$.
Define sets $V_k\subset\reals^{d-1}$ by setting $V_0=B(0,r)$
and
\[ V_{k+1} = \scriptc\cap (t^{-1}V_k-t^{-1}(1-t)B(0,r))\cap ((1-t)^{-1}V_k-(1-t)^{-1}tB(0,r)).\]
From the facts that $B(0,r)\subset\scriptc$, $\scriptc$ is convex, $r\gtrsim 1$, and $|\scriptc|=O(1)$,
it follows easily that there exists a positive integer $N$, depending only on the dimension $d$
and on the compact subset $\Lambda\subset(0,1)$ in which $t$ is assumed to lie,
such that $V_N=\scriptc$.

Suppose we knew that \[t\varphi(x')+(1-t)\psi(x'')\equiv \zeta(tx'+(1-t)x'')\] for all $x',x''\in\scriptc$,
and that $\varphi,\psi,\zeta=O(1)$ on $V_k$.
We could then 
conclude that $\varphi(x')=O(1)$ for all $x'\in \scriptc$ for which there exists $x''\in\scriptc$
such that $tx'+(1-t)x''\in V_k$ and $x''\in B(0,r)$.
By writing \[x' = t^{-1}(tx'+(1-t)x'') - t^{-1}(1-t)x'',\]
and noting that $V_k,B(0,r)$ are convex sets which contain $0$, 
we see that for any point $x'\in t^{-1}V_k-t^{-1}(1-t)B(0,r)$ there exists such an $x''$.
Thus provided that $x'\in\scriptc$, $\varphi(x')=O(1)$. By the same reasoning, 
\[
\psi(x'')=O(1) \text{ for all  }
x''\in \scriptc\cap \big((1-t)^{-1}V_k-(1-t)^{-1}tB(0,r)\big),\] so $\varphi,\psi=O(1)$ on $V_{k+1}$.
Since $V_{k+1}$ is convex, the same conclusion then follows for $\zeta$ via the
approximate functional equation for $\zeta(tx'+(1-t)x'')$.
This argument can be iterated $N$ times to yield the conclusion for $V_N=\scriptc$.

It is elementary to verify that this same procedure applies in the context at hand, in which
the approximate functional equation is known to hold for the vast majority of all pairs of points,
in the sense of Lebesgue measure. Details are omitted.
\end{proof}

\section{Conclusion of proof}

The remainder of the proof of Theorem~\ref{thm:BMadd} follows the reasoning already given for
the case $d=2$ in \cite{christbmtwo}, steps 8,9,10, with no essential changes. 
The main points are (i) that the above reasoning can and should be applied to
all simultaneous rotations of $A,B$, 
and (ii) that the information thus obtained makes possible a compactness argument,
which, when coupled with the known fact that exact equality can hold only for homothetic pairs of 
convex sets, yields the conclusion of the theorem.

\begin{thebibliography}{20}

\bibitem{burchard} A.~Burchard,
{\em Cases of equality in the Riesz rearrangement inequality}, 
Ann. of Math. (2) 143 (1996), no. 3, 499--527

\bibitem{youngdiscrete}
M.~Charalambides and M.~Christ,
{\em Near--extremizers for Young's inequality for discrete groups},
preprint, math.CA arXiv:1112.3716

\bibitem{christradon} M.~Christ, 
{\em Extremizers of a Radon transform inequality}, preprint  math.CA arXiv:1106.0719,
to appear in proceedings of conference in honor of E.~M.~Stein 

\bibitem{christrieszsobolev} \bysame
{\em An approximate inverse Riesz-Sobolev rearrangement inequality}, preprint,
math.CA arXiv:1112.3715

\bibitem{christyoungest} \bysame
{\em On near-extremizers for Young's inequality for $\reals^d$}, preprint,
math.CA arXiv:1112.4875

\bibitem{christbmtwo} \bysame
{\em Near equality in the two-dimensional Brunn-Minkowski inequality},
math.CA arXiv:1206.1965

\bibitem{gardner}
R.~J.~Gardner, {\em The Brunn-Minkowski inequality}, 
Bull. Amer. Math. Soc. (N.S.) 39 (2002), no. 3, 355--405.

\bibitem{HO}
H.~Hadwiger and D.~Ohmann, 
{\em Brunn-Minkowskischer Satz und Isoperimetrie},  
Math. Z. 66 (1956), 1--8

\bibitem{henstock}
R.~Henstock and A.~M.~ Macbeath, 
{\em On the measure of sum-sets. I. The theorems of Brunn, Minkowski, and Lusternik},
Proc. London Math. Soc. (3) 3, (1953), 182--194



\bibitem{liebloss}
E.~H.~Lieb and M.~Loss, 
{\em Analysis}, Amer. Math. Soc., Providence, RI, 1997

\bibitem{riesz}
F.~Riesz,
{\em Sur une in\'egalit\'e int\'egrale}, Journal London Math. Soc. 5 (1930), 162--168



\bibitem{sobolev}
S.~L.~Sobolev,
{\em On a theorem of functional analysis}, Mat. Sb. (N.S.) 4 (1938), 471--479,
A. M. S. Transl.  Ser. 2, 34 (1963), 39-68

\bibitem{taovu}
T.~Tao and V.~Vu,
{\em Additive Combinatorics},
Cambridge Studies in Advanced Mathematics, 105. Cambridge University Press, Cambridge, 2006

\end{thebibliography}
\end{document}